\documentclass[11pt,letterpaper]{article} 

\usepackage{amsmath}
\usepackage{amssymb}
\usepackage{bbm}

\usepackage[margin=1in]{geometry}
\usepackage[utf8]{inputenc}
\usepackage{amsthm}
\usepackage{thm-restate}
\newtheorem{theorem}{Theorem}
\newtheorem{lemma}[theorem]{Lemma}

\newtheorem{proposition}[theorem]{Proposition}
\newtheorem{definition}[theorem]{Definition}

\newtheorem{example}[theorem]{Example}

\usepackage[pdfpagelabels,bookmarks=false]{hyperref}
\hypersetup{colorlinks, linkcolor=darkblue, citecolor=darkgreen, urlcolor=darkblue}
\usepackage[capitalize, nameinlink]{cleveref}

\crefname{theorem}{Theorem}{Theorems}
\Crefname{lemma}{Lemma}{Lemmas}
\Crefname{proposition}{Proposition}{Propositions}
\Crefname{claim}{Claim}{Claims}
\Crefname{fact}{Fact}{Facts}
\Crefname{remark}{Remark}{Remarks}
\Crefname{observation}{Observation}{Observations}
\Crefname{line}{Line}{Lines}
\crefalias{AlgoLine}{line}
\Crefname{algocf}{Algorithm}{Algorithms}

\usepackage{times}
\usepackage[table]{xcolor}

\usepackage{url}
\usepackage{array}
\usepackage{setspace}

\usepackage{lmodern}
\usepackage{mathtools}
\usepackage{bm}
\usepackage{xspace}

\usepackage{xfrac}

\usepackage{tikz}
\usetikzlibrary{calc}
\usetikzlibrary{shapes.geometric}
\usetikzlibrary{arrows}
\usetikzlibrary{decorations.pathreplacing, decorations.pathmorphing}

\usepackage{float}
\usepackage{graphicx}

\usepackage{nicematrix}

\usepackage[vlined,ruled,algo2e,linesnumbered]{algorithm2e}
\SetAlgoSkip{bigskip}
\SetAlgoInsideSkip{smallskip}
\SetAlCapHSkip{0.2em}
\SetCustomAlgoRuledWidth{0.95\textwidth}

\definecolor{darkblue}{rgb}{0,0,0.38}
\definecolor{darkred}{rgb}{0.8,0,0}
\definecolor{darkgreen}{rgb}{0.1,0.35,0}

\renewcommand{\epsilon}{\varepsilon}

\newcommand{\Cscr}{\mathcal{C}}

\newcommand{\Fscr}{\mathcal{F}}

\newcommand{\Tscr}{\mathcal{T}}
\newcommand{\Uscr}{\mathcal{U}}

\newcommand{\LP}{\mathrm{LP}}

\newcommand{\exc}{\theta}

\title{Cost Allocation for Set Covering: the Happy Nucleolus}

\date{\small Research Institute for Discrete Mathematics and Hausdorff Center for Mathematics, University of Bonn \\
\scriptsize \texttt{\{blauth,ellerbrock,traub,vygen\}@dm.uni-bonn.de}}

  \author{
    Jannis Blauth
    \and
    Antonia Ellerbrock
    \and
    Vera Traub
    \and
    Jens Vygen
    }

\begin{document}

\maketitle

\begin{abstract}
We consider cost allocation for set covering problems. We allocate as much cost to the elements (players) as possible
without violating the group rationality condition (no subset of players pays more than covering this subset would cost), 
and so that the excess vector is lexicographically maximized.
This is identical to the well-known nucleolus if the core of the corresponding cooperative game is nonempty, 
i.e., if some optimum fractional cover is integral.
In general, we call this the \emph{happy nucleolus}.

Like for the nucleolus, the excess vector contains an entry for every subset of players, not only for the sets in
the given set covering instance. 
Moreover, it is NP-hard to compute a single entry because this requires solving a set covering problem.
Nevertheless, we give an explicit family of at most $mn$ subsets, each with a trivial cover (by a single set), 
such that the happy nucleolus is always completely determined by this proxy excess vector;
here $m$ and $n$ denote the number of sets and the number of players in our set covering instance. 
We show that this is the unique minimal such family in a natural sense.

While computing the nucleolus for set covering is NP-hard, our results imply that the happy nucleolus can be computed in polynomial time.\\
\end{abstract}

\section{Introduction}

Cost allocation has been studied in various settings; we 
consider it in the context of the classical set covering problem. Given a finite set $P$ (of players) and a family $\Tscr$ of subsets of $P$, 
with nonnegative costs $c(T)$ for $T\in\Tscr$, 
we ask for a minimum-cost sub-family $\Uscr\subseteq\Tscr$ such that $\Uscr$ \emph{covers} $P$, 
i.e., every element of $P$ belongs to some set in $\Uscr$.
We may assume that $\Tscr$ covers $P$, i.e., the union of all sets in $\Tscr$ is $P$.
The problem of cost allocation now asks to allocate a certain amount $y_p\ge 0$ of the total cost to each player $p\in P$, subject to certain efficiency and fairness conditions.

Different criteria of fairness can be defined.
A nonempty subset of players is called a \emph{coalition}.
Call a coalition $S$ \emph{happy} if $y(S):=\sum_{p\in S}y_p$ is at most the cost of covering $S$, i.e.,
the \emph{excess} $\min\{c(\Uscr): \Uscr\subseteq\Tscr,\, \Uscr \text{ covers } S\} - y(S)$ is nonnegative, where $c(\Uscr) \coloneqq  \sum_{T \in \Uscr} c(T)$ for $\Uscr \subseteq \Tscr$.
The \emph{group rationality} condition requires all coalitions to be happy.
\emph{Individual rationality} means group rationality for all singletons.

The set of cost allocations that satisfy group rationality and allocate the cost of an optimum set cover is called the \emph{core} \cite{gillies1959solutions}.
However, the core is often empty for set covering instances \cite{tamir1989core, potters1992traveling}. Then, one can either violate group rationality 
or allocate less. 
The maximum total amount that can be allocated to the players subject to group rationality is the
minimum cost of a \emph{fractional} set cover (\cite{deng1999algorithmic}, \cref{prop:lp_allocated}). 

Caprara and Letchford \cite{caprara2010new} mentioned applications of such cost allocations. In one example, a public service is provided,
and as large a part as possible of the total cost shall be distributed to the consumers such that no coalition pays more than the service just for this coalition would cost. 
In other situations, all the cost must be distributed to the consumers; then, one can scale up the cost allocation by a factor $\gamma\ge 1$, 
which may still be acceptable if $\gamma$ is not too big because no coalition pays more than $\gamma$ times the amount 
that the service just for this coalition would cost. 

There can be many different cost allocations that maximize the allocated cost subject to group rationality. 
Even players that are completely symmetric can receive different values, which is often unacceptable in practice.
Consider the example in \cref{fig:simpleexample}, which shows a simple instance of a vehicle routing problem.
This can be viewed as a set covering problem where every possible tour corresponds to the set of customers it serves. Vehicle routing problems will serve as an example for us several times.

\begin{figure}
\begin{center}
        \begin{tikzpicture}[scale=1.5, thick]
        \tikzstyle{customer}=[circle, draw, scale=.7, minimum size=22pt, inner sep=1]
        \begin{scope}[xshift=0]
            \node [scale=1.5, rectangle, draw, fill=black] (D) at (0,0) {};
            \node [customer] (c) at (1,0) {\Large c};
            \node [customer] (d) at (2,0) {\Large d};
            \node [customer] (e) at (3,0) {\Large e};
            \node [customer] (a) at (-0.8,0.6) {\Large a};
            \node [customer] (b) at (-0.8,-0.6) {\Large b};
            \node[below=2mm] at (c) {\bf 2};
            \node[below=2mm] at (d) {\bf 0};
            \node[below=2mm] at (e) {\bf 4};
            \node[left=2mm] at (a) {\bf 2};
            \node[left=2mm] at (b) {\bf 1};
            \draw[color=blue] (D) -- (c) node[midway,above] {1};
            \draw[color=blue] (c) -- (d) node[midway,above] {1};
            \draw[color=blue] (d) -- (e) node[midway,above] {1};
            \draw[color=blue] (D) -- (a) node[midway,above right=-1mm] {1};
            \draw[color=blue] (a) -- (b) node[midway,left] {1};
            \draw[color=blue] (b) -- (D) node[midway,below right=-1mm] {1};
            \end{scope}
            \begin{scope}[xshift=5.5cm]
            \node [scale=1.5, rectangle, draw, fill=black] (D) at (0,0) {};
            \node [customer] (c) at (1,0) {\Large c};
            \node [customer] (d) at (2,0) {\Large d};
            \node [customer] (e) at (3,0) {\Large e};
            \node [customer] (a) at (-0.8,0.6) {\Large a};
            \node [customer] (b) at (-0.8,-0.6) {\Large b};
            \node[below=2mm] at (c) {\bf 1};
            \node[below=2mm] at (d) {\bf 1.5};
            \node[below=2mm] at (e) {\bf 3.5};
            \node[left=2mm] at (a) {\bf 1.5};
            \node[left=2mm] at (b) {\bf 1.5};
            \draw[color=blue] (D) -- (c) node[midway,above] {1};
            \draw[color=blue] (c) -- (d) node[midway,above] {1};
            \draw[color=blue] (d) -- (e) node[midway,above] {1};
            \draw[color=blue] (D) -- (a) node[midway,above right=-1mm] {1};
            \draw[color=blue] (a) -- (b) node[midway,left] {1};
            \draw[color=blue] (b) -- (D) node[midway,below right=-1mm] {1};
            \end{scope}
        \end{tikzpicture}
\end{center}
\caption{\small This simple vehicle routing instance asks for a set of tours that visit all customers a, b, c, d, e; 
here a \emph{tour} is a cycle that starts and ends at the depot, the black square. 
Distances are given by the metric closure of the numbers on depicted edges.
The optimum vehicle routing solution costs 9, and in this example we can allocate 9 units and satisfy group rationality.
Two possible such cost allocations are shown in bold next to the vertices. 
The one on the left is not even symmetric: although a and b are completely symmetric, they pay different shares.
Moreover, customer c pays more than customer d although it is definitely not harder to serve.
The allocation on the right looks fairer, and this is in fact the (happy) nucleolus. 
\label{fig:simpleexample}}
\end{figure}
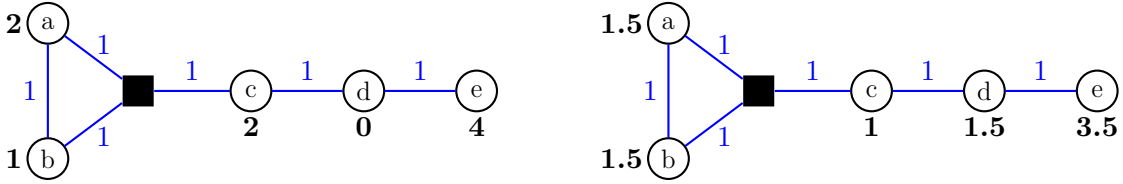

\subsection{The happy nucleolus} \label{subsec:happy_nuc}

To obtain a fair cost allocation, Schmeidler defined the \emph{nucleolus} \cite{schmeidler1969nucleolus}.
It is the unique cost allocation that distributes the entire cost, fulfills individual rationality, and lexicographically maximizes the \emph{excess vector};
this vector contains the excess of every coalition, and these entries are sorted in nondecreasing order.\footnote{
For set covering instances, we could drop the individual rationality condition because it is automatically satisfied. This follows from 5.(1) in \cite{potters1992nucleolus} because the minimum cost to cover a coalition is subadditive. 
}

\begin{definition}[excess vector]
Let $(P,\Tscr,c)$ be a set covering instance. 
Let $\Cscr$ denote the set of pairs $(S,\Uscr)$ such that $\emptyset\not=S\subseteq P$, $\Uscr\subseteq\Tscr$, and $\Uscr$ covers $S$.
Let $y\in\mathbb{R}_{\ge 0}^P$. 
The excess of $(S,\Uscr)\in\Cscr$ is
\begin{equation*}
    \exc^y(S,\Uscr) \ := \ c(\Uscr)-y(S) \enspace .
\end{equation*}
The excess of a coalition $S$ is defined as 
\begin{equation*}
    \exc^y(S) \ := \ \min \left\{\big. \exc^y(S,\Uscr):\Uscr\subseteq\Tscr,\, \Uscr \text{ covers } S \right\} \enspace .
\end{equation*}
The \emph{excess vector} $\exc^y$ contains the entries $\exc^y(S)$ for all coalitions $S$ in nondecreasing order.
\end{definition}

By definition, a vector $y\in\mathbb{R}_{\ge 0}^P$ satisfies group rationality if and only if its excess vector is nonnegative.
The nucleolus satisfies group rationality if and only if the core is nonempty. 
The right-hand side of \cref{fig:simpleexample} shows an example.

To obtain a fair cost allocation that satisfies group rationality even when the core is empty, we introduce the concept of the
happy nucleolus. Recall that we want to allocate as much as possible subject to group rationality; this amount 
is the minimum cost of a fractional set cover:
\begin{align*}
\LP(P,\Tscr,c) \ := \
\min\, \left\{\sum_{T\in\Tscr} c(T) x_T \  : \ x\in\mathbb{R}_{\ge 0}^{\Tscr}, \ \sum_{T\in\Tscr : \, p\in T} x_T \ge 1 \text{ for all } p\in P \right\} \enspace .
\end{align*}

\begin{proposition}[\cite{deng1999algorithmic}] \label{prop:lp_allocated}
$\LP(P,\Tscr,c) =  \max \left\{ y(P) \ : \ y\in\mathbb{R}_{\ge 0}^P \text{ satisfies group rationality} \right\}$. 
\end{proposition}

\begin{proof} 
We observe
\begin{align*}
\LP(P,\Tscr,c) \ = \ & \max\, \left\{ y(P) \ : \ y\in\mathbb{R}_{\ge 0}^P, \ y(T)\le c(T) \text{ for all } T\in\Tscr \right\} \\
\ = \ & \max\, \left\{ y(P) \ : \ y\in\mathbb{R}_{\ge 0}^P, \ y(S)\le c(\Uscr) \text{ for all } S\subseteq P \text{ and all } \Uscr\subseteq\Tscr \text{ covering } S \right\} .
\end{align*}
Here, the first equation is LP duality. For the second equation, ``$\ge$'' is trivial. 
To show ``$\le$'', let $y\in\mathbb{R}_{\ge 0}^P$ with $y(T)\le c(T)$ for all $T\in\Tscr$. Then,
\begin{equation*}
    y(S) \ \le \ \sum_{U\in\Uscr} y(U) \ \le \ \sum_{U\in\Uscr} c(U) \ = \ c(\Uscr)
\end{equation*}
for all $S\subseteq P$ and $ \Uscr\subseteq\Tscr$ covering $S$.
\end{proof} 

Now we can define:

\begin{definition}[happy nucleolus]
Let $(P,\Tscr,c)$ be a set covering instance.
A vector $y$ is called \emph{happy nucleolus} of $(P,\Tscr, c)$ if it lexicographically maximizes the excess vector $\exc^y$
among all vectors $y\in\mathbb{R}_{\ge 0}^P$ with $y(P)=\LP(P,\Tscr,c)$.
\end{definition}

\begin{figure}
\begin{center}
    \begin{tikzpicture}[scale=1.45, thick]
     \tikzstyle{customer}=[circle, draw, scale=.7, minimum size=22pt, inner sep=1]
        \begin{scope}[xshift=0]
            \node[scale=1.5, rectangle, draw, fill=black] (D) at (-.1,0) {};
            \node[customer] (a) at (-1.7,0.5) {\Large a};
            \node[customer] (b) at (-1,0) {\Large b};
            \node[customer] (c) at (-1.7,-0.5) {\Large c};
            \node[customer] (d) at (.7,0) {\Large d};
            \draw[color=blue] (D) -- (b) node[midway,above] {4};
            \draw[color=blue] (D) -- (d) node[midway,above] {3};
            \draw[color=blue] (a) -- (b) node[near end,above] {0};
            \draw[color=blue] (a) -- (c) node[midway,left] {0};
            \draw[color=blue] (b) -- (c) node[near start, below] {0};
            \path[color=blue, bend left=50] (b) edge (d) node[midway, above=.4cm] {\scriptsize 6};
            \node[left=2mm] at (a) {\bf 5};
            \node[right=3mm, below=1mm] at (b) {\bf 5};
            \node[left=2mm] at (c) {\bf 5};
            \node[right=2mm] at (d) {\bf 6};
        \end{scope}
        \begin{scope}[xshift=4cm]
            \node [scale=1.5, rectangle, draw, fill=black] (D) at (-0.1,0) {};
            \node [customer] (a) at (-1.7,0.5) {\Large a};
            \node [customer] (b) at (-1,0) {\Large b};
            \node [customer] (c) at (-1.7,-0.5) {\Large c};
            \node [customer] (d) at (.7,0) {\Large d};
            \draw[color=blue] (D) -- (b) node[midway,above] {4};
            \draw[color=blue] (D) -- (d) node[midway,above] {3};
            \draw[color=blue] (a) -- (b) node[near end,above] {0};
            \draw[color=blue] (a) -- (c) node[midway,left] {0};
            \draw[color=blue] (b) -- (c) node[near start, below] {0};
            \path[color=blue, bend left=50] (b) edge (d) node[midway, above=.4cm] {\scriptsize 6};
            \node[left=2mm] at (a) {\bf 4.6};
            \node[right=4mm, below=1mm] at (b) {\bf 4.6};
            \node[left=2mm] at (c) {\bf 4.6};
            \node[right=2mm] at (d) {\bf 7.2};
        \end{scope}
        \begin{scope}[xshift=8cm]
            \node [scale=1.5, rectangle, draw, fill=black] (D) at (-.1,0) {};
            \node [customer] (a) at (-1.7,0.5) {\Large a};
            \node [customer] (b) at (-1,0) {\Large b};
            \node [customer] (c) at (-1.7,-0.5) {\Large c};
            \node [customer] (d) at (.7,0) {\Large d};
            \draw[color=blue] (D) -- (b) node[midway,above] {4};
            \draw[color=blue] (D) -- (d) node[midway,above] {3};
            \draw[color=blue] (a) -- (b) node[near end,above] {0};
            \draw[color=blue] (a) -- (c) node[midway,left] {0};
            \draw[color=blue] (b) -- (c) node[near start, below] {0};
            \path[color=blue, bend left=50] (b) edge (d) node[midway, above=.4cm] {\scriptsize 6};
            \node[left=2mm] at (a) {\bf 4};
            \node[right=2.5mm, below=1mm] at (b) {\bf 4};
            \node[left=2mm] at (c) {\bf 4};
            \node[right=2mm] at (d) {\bf 6};
        \end{scope}
    \end{tikzpicture}

    \vspace{.5cm}
    
    { \onehalfspacing \small
    \begin{tabular}{|l!{\vrule width 1.1pt}c|c|c|c|c|c|c!{\vrule width 1.1pt}l|}
        \hline
        \rowcolor{gray!10} \textbf{Excess} & \bm{$\{a\}$} & \bm{$\{d\}$} & \bm{$\{a,b\}$} & \bm{$\{a,d\}$} & \bm{$\{a,b,c\}$} & \bm{$\{a,b,d\}$} & \bm{$\{a,b,c,d\}$} & \\ \noalign{\hrule height 1.1pt}
        \cellcolor{gray!10} \textbf{Left} & 3 & 0 & \cellcolor{green!20}$-2$ & 2 & 1 & \cellcolor{green!20}$-2$ & 0 & \cellcolor{gray!10} \textbf{Nucleolus} \\ \hline
        \cellcolor{gray!10} \textbf{Middle} & 3.4 & \cellcolor{green!20}$-1.2$ & \cellcolor{green!20}$-1.2$ & 1.2 & \cellcolor{red!20}2.2 & \cellcolor{red!20}$-2.4$ & \cellcolor{red!20} 0 & \cellcolor{gray!10} \\ \hline
        \cellcolor{gray!10} \textbf{Right} & 4 & \cellcolor{green!20}0 & \cellcolor{green!20}0 & 3 & 4 & \cellcolor{green!20}0 & 3 & \cellcolor{gray!10} \textbf{Happy Nucleolus} \\ \hline
    \end{tabular}
    }
\end{center}
\caption{\small This capacitated vehicle routing instance asks for a set of tours that visit all customers a, b, c, d; 
here one tour can serve up to two customers. 
Again, distances are given by the metric closure of the numbers on depicted edges.  
Optimum integral and fractional vehicle routing solutions have total cost 21 and 18, respectively. 
The cost allocation shown on the left is the nucleolus, the one on the right is the happy nucleolus. 
The table shows all excess values (up to symmetric coalitions); green fields highlight the first entries of the excess vector.
The middle solution allocates the same amount as the nucleolus but optimizes the excess only for coalitions that can be served by a single tour
(red fields indicate ignored coalitions). 
Note that this solution differs from the nucleolus. The happy nucleolus would stay the same if we considered only coalitions served by a single tour. 
\label{fig:multtoursexp}}
\end{figure}

By \cref{prop:lp_allocated}, the excess vector of a happy nucleolus is nonnegative.
For the definition of the happy nucleolus, only one entry for each coalition, corresponding to a minimum-cost cover, is relevant.
Hence, the happy nucleolus can also be viewed as the nucleolus of the cooperative game $(P,\bar c)$ with 
 \[
\bar c(S) \ := \ \begin{cases} \min\, \left\{c(\Uscr) : \Uscr\subseteq\Tscr,\, \Uscr \text{ covers } S \right\} & \text{ if } S\neq P \\
\LP(P,\Tscr,c) & \text{ if } S=P \end{cases}
\]
for $S\subseteq P$; see also \cite{aziz2010monotone}. 
Schmeidler \cite{schmeidler1969nucleolus} proved that the nucleolus of any cooperative game is unique;
hence, also the happy nucleolus is unique.
Uniqueness will also follow from \cref{thm:few_constraints_determine_nucleolus} below. 
An immediate consequence of uniqueness is that symmetric players get the same value.

Another interesting property is that the happy nucleolus does not depend on coalitions $S \subseteq P$ 
whose minimum-cost cover requires more than one set from $\mathcal{T}$, 
a property that follows directly from prior work \cite{huberman1980nucleolus} (and from \cref{thm:few_constraints_determine_nucleolus} below),
but which does not hold for the (unhappy) nucleolus; an example can be found in \cite{chardaire2001core}. 
\cref{fig:multtoursexp} also illustrates this in the setting of capacitated vehicle routing. 
The property of the happy nucleolus that is crucial here compared to the (unhappy) nucleolus is the nonnegativity of its excess vector which ensures that summed excess values will be larger than their summands and thus less relevant for lexicographic excess maximization.

\cref{fig:first_random_example} shows a more complex example of the happy nucleolus in the context of vehicle routing.

\colorlet{darkgreen}{green!80!black}
\colorlet{color_a}{blue!28.5!darkgreen}
\colorlet{color_b}{blue!54!darkgreen}
\colorlet{color_c}{blue!8!darkgreen}
\colorlet{color_d}{blue!38.5!darkgreen}
\colorlet{color_e}{darkgreen!97.5!yellow}
\colorlet{color_f}{darkgreen!45.5!yellow}
\colorlet{color_g}{yellow!69!red}
\colorlet{color_h}{yellow!88!red}
\colorlet{color_i}{yellow!50!red}
\colorlet{color_j}{blue}
\colorlet{color_k}{darkgreen!99!yellow}
\colorlet{color_l}{yellow!15!red}
\colorlet{color_m}{darkgreen!91!yellow}
\colorlet{color_n}{yellow!57.5!red}
\colorlet{color_o}{yellow!74.5!red}

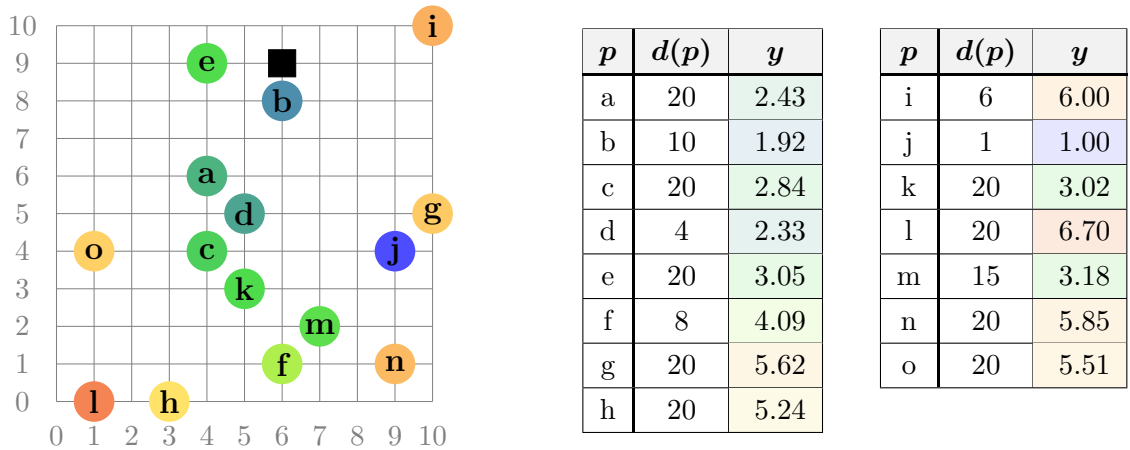
\begin{figure}[htb]
	\begin{minipage}{0.45\textwidth}
		\begin{center}
			\begin{tikzpicture}[scale=.5, every node/.style = {minimum size = 0.9cm, inner sep = 0pt}]
				
				\foreach \i in {0,...,10} {
					\draw [very thin,gray] (\i,0) -- (\i,10)  node [below] at (\i,0) {$\i$};
				}
				\foreach \i in {0,...,10} {
					\draw [very thin,gray] (0,\i) -- (10,\i) node [left] at (0,\i) {$\i$};
				}
				
				\node [scale=.4, rectangle, draw, fill=black] (D) at (6,9) {};
				\node [circle, fill=color_a!70, scale=.6] (1) at (4,6) {\huge \textnormal{\textbf{\textcolor{black}{a}}}};
				\node [circle, fill=color_b!70, scale=.6] (2) at (6,8) {\huge \textnormal{\textbf{\textcolor{black}{b}}}};
				\node [circle, fill=color_c!70, scale=.6] (3) at (4,4) {\huge \textnormal{\textbf{\textcolor{black}{c}}}};
				\node [circle, fill=color_d!70, scale=.6] (4) at (5,5) {\huge \textnormal{\textbf{\textcolor{black}{d}}}};
				\node [circle, fill=color_e!70, scale=.6] (5) at (4,9) {\huge \textnormal{\textbf{\textcolor{black}{e}}}};
				\node [circle, fill=color_f!70, scale=.6] (6) at (6,1) {\huge \textnormal{\textbf{\textcolor{black}{f}}}};
				\node [circle, fill=color_g!70, scale=.6] (7) at (10,5) {\huge \textnormal{\textbf{\textcolor{black}{g}}}};
				\node [circle, fill=color_h!70, scale=.6] (8) at (3,0) {\huge \textnormal{\textbf{\textcolor{black}{h}}}};
				\node [circle, fill=color_i!70, scale=.6] (9) at (10,10) {\huge \textnormal{\textbf{\textcolor{black}{i}}}};
				\node [circle, fill=color_j!70, scale=.6] (10) at (9,4) {\huge \textnormal{\textbf{\textcolor{black}{j}}}};
				\node [circle, fill=color_k!70, scale=.6] (11) at (5,3) {\huge \textnormal{\textbf{\textcolor{black}{k}}}};
				\node [circle, fill=color_l!70, scale=.6] (12) at (1,0) {\huge \textnormal{\textbf{\textcolor{black}{l}}}};
				\node [circle, fill=color_m!70, scale=.6] (13) at (7,2) {\huge \textnormal{\textbf{\textcolor{black}{m}}}};
				\node [circle, fill=color_n!70, scale=.6] (14) at (9,1) {\huge \textnormal{\textbf{\textcolor{black}{n}}}};
				\node [circle, fill=color_o!70, scale=.6] (15) at (1,4) {\huge \textnormal{\textbf{\textcolor{black}{o}}}};
			\end{tikzpicture}
		\end{center}
	\end{minipage}
	\begin{minipage}{0.55\textwidth}
		\begin{center}
			\onehalfspacing
			
			\begin{tabular}[t]{|c!{\vrule width 1.1pt}c|c|}
				\hline
				\rowcolor{gray!10} \bm{$p$} & \bm{$d(p)$} & \bm{$y$} \\ \noalign{\hrule height 1.1pt}
				a & 20 & \cellcolor{color_a!10} 2.43 \\ \hline
				b & 10 & \cellcolor{color_b!10} 1.92 \\ \hline
				c & 20 & \cellcolor{color_c!10} 2.84 \\ \hline
				d & 4 & \cellcolor{color_d!10} 2.33 \\ \hline
				e & 20 & \cellcolor{color_e!10} 3.05 \\ \hline
				f & 8 & \cellcolor{color_f!10} 4.09 \\ \hline
				g & 20 & \cellcolor{color_g!10} 5.62 \\ \hline
				h & 20 & \cellcolor{color_h!10} 5.24 \\ \hline
			\end{tabular}
			\hspace{.5cm}
			\begin{tabular}[t]{|c!{\vrule width 1.1pt}c|c|}
				\hline
				\rowcolor{gray!10} \bm{$p$} & \bm{$d(p)$} & \bm{$y$} \\ \noalign{\hrule height 1.1pt}
				i & 6 & \cellcolor{color_i!10} 6.00 \\ \hline
				j & 1 & \cellcolor{color_j!10} 1.00 \\ \hline
				k & 20 & \cellcolor{color_k!10} 3.02 \\ \hline
				l & 20 & \cellcolor{color_l!10} 6.70 \\ \hline
				m & 15 & \cellcolor{color_m!10} 3.18 \\ \hline
				n & 20 & \cellcolor{color_n!10} 5.85 \\ \hline
				o & 20 & \cellcolor{color_o!10} 5.51 \\ \hline
			\end{tabular}
		\end{center}
	\end{minipage}
	\caption{\small A capacitated vehicle routing instance with 15 customers, each with coordinates in $\{0,\ldots,10\}^2$. 
		One tour can serve up to five customers; all tours must begin and end at the depot (black square, at (6,9)). 
		The cost of a tour is the total Euclidean distance traveled. We can also decide not to serve a customer $p$
		and instead pay a dropping penalty $d(p)$, shown in the second column of the table.
		The third column shows (rounded values of) the happy nucleolus $y$, visualized by the color scale \includegraphics[width=.1\textwidth]{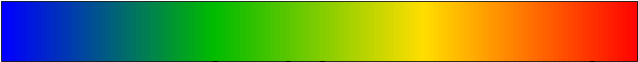} from 1 to 7.
		\label{fig:first_random_example}}
\end{figure}

\subsection{Related work}

Based on prior work about cost allocation in cooperative game theory \cite{gillies1959solutions, aumann1961bargaining, davis1965kernel}, Schmeidler defined the nucleolus \cite{schmeidler1969nucleolus}.
Axiomatizations for the nucleolus and related concepts were provided in \cite{kohlberg1971nucleolus, sobolev1975characterization,peleg1986reduced,potters1991axiomatization,snijders1995axiomatization}. 

The nucleolus can be computed by the \textit{MPS scheme} \cite{maschler1979geometric}, 
which is sometimes called Maschler's scheme although Kopelowitz \cite{kopelowitz1967computation} attributed it to Peleg.
In special cases, the algorithm was implemented, e.g.\@ for vehicle routing in \cite{gothe1996nucleolus, engevall2004heterogeneous}.
In some cases, a (pseudo-)polynomial running time could be obtained \cite{caprara2010new, konemann2020computing, konemann2020general, pashkovich2021computing}.
These special cases include neither set covering nor vehicle routing. 
Chen, Lu and Zhang \cite{chen2012computing} devised an algorithm for the nucleolus of weighted vertex cover games on bipartite graphs that runs polynomially in the number of players.
In many other cases,  the nucleolus is NP-hard to compute if the cost function is given in a compact form (not listing the cost for every coalition explicitly). 
This includes spanning tree games, where an early proof of NP-hardness was provided by Faigle, Kern and Kuipers \cite{faigle1998note}. 
See also \cite{greco2015complexity} for an overview.

Since the core is often empty, various relaxations have been proposed.
The set of cost allocations that distribute the total cost and violate all group rationality constraints by at most an additive constant $\varepsilon$ has often been called the \emph{$\varepsilon$-core} \cite{shapley1966quasi}.
For the minimum $\varepsilon$ for which this is nonempty, the $\varepsilon$-core is called the \emph{least core} \cite{maschler1979geometric}.
Therefore, the nucleolus is part of the least core.
One could also relax group rationality multiplicatively by a factor $(1+\varepsilon)$, resulting in the \emph{$\varepsilon$-approximate core} \cite{faigle1998approximately}. 
This is equivalent to the option that we will follow, 
which is to insist on group rationality 
and distribute a $\gamma$-fraction of the total cost \cite{jain2007cost}. This approach has also been followed by \cite{bachrach2009cost, caprara2010new, liu2016computing}.
The happy nucleolus is such a cost allocation for the maximum possible $\gamma$.

Completely different concepts for cost allocation include the \emph{Shapley value} \cite{shapley1953value} and the \emph{$\tau$-value} \cite{tijs1981bounds}.
For vehicle routing, these and other concepts have been analyzed in \cite{krajewska2008horizontal, frisk2010cost, ozener2013allocating, dahlberg2015cooperative, agussurja2016achieving, dahlberg2018consolidation}.
Guajardo and R\"onnqvist \cite{guajardo2016review} and Schulte, Lalla-Ruiz, Schwarze, Gonz{\'a}lez-Ram{\'i}rez, and Vo{\ss} \cite{schulte2019scalable} wrote surveys about cost allocation in vehicle routing and TSP.

None of the previously proposed concepts can guarantee all of the following three desirable properties:
\begin{itemize}
\item group rationality,
\item symmetry, i.e., identical copies of customers get the same share,
\item independence on disjoint subinstances, i.e., if we have two instances on two disjoint player sets $P_1$ and $P_2$, then the cost allocations of the separate instances also form the cost allocation of the united instance.\footnote{This also holds for the nucleolus and is a direct consequence of the \emph{reduced game property} \cite{sobolev1975characterization}.}
\end{itemize}
The happy nucleolus satisfies all three properties. The third one follows from the above-mentioned property that coalitions whose minimum-cost cover uses multiple sets from $\mathcal{T}$ can be ignored when determining the happy nucleolus. We will further strengthen this in the following section (cf.\ \cref{thm:few_constraints_determine_nucleolus}).

\subsection{Main results}

The happy nucleolus can be computed by the MPS scheme, designed for computing the nucleolus \cite{maschler1979geometric, kopelowitz1967computation}.
It maximizes the smallest excess by solving a linear program,
computes all coalitions whose excess is always among the smallest excess values, fixes the total contribution $y(S)$ of such coalitions $S$,
and iterates by maximizing the next smallest excess. 
This is computationally very expensive if all coalitions need to be considered.
Computational details of the MPS scheme can be found in \cite{nguyen2016finding}.
Furthermore, it is already NP-hard to compute the excess of a single coalition (this requires solving a set covering problem).

However, once we have the (happy) nucleolus of a set covering instance $(P,\Tscr,c)$, it is actually determined by only $2|P|-2$ coalitions \cite{reijnierse1998b}.
This bound is tight as we will show in \cref{thm:aposteriorilowerbound}. 
Unfortunately, those coalitions are known only \emph{a posteriori}.  
We will show that \emph{a priori}, one can restrict all computations to fewer than $|P|\cdot|\Tscr|$ coalitions, each with a trivial cover
(by a single set from $\Tscr$). We call a pair $(S,\Uscr)\in\Cscr$ \emph{simple} if $|\Uscr|=1$ and
denote by 
\[
\Cscr^{\textnormal{simple}} \ := \ \{(S,\Uscr)\in\Cscr: |\Uscr|=1\}
\] 
the family of simple pairs. 

\begin{definition}[determining the happy nucleolus]
Let $(P,\Tscr,c)$ be a set covering instance.
For a vector $y\in\mathbb{R}^P_{\ge 0}$ and a subset $\Cscr'\subseteq\Cscr$, the $\Cscr'$-excess vector $\exc^y[\Cscr']$ 
contains the entries $\exc^y(S,\Uscr)$ for $(S,\Uscr)\in\Cscr'$ in nondecreasing order.
We say that a subset $\Cscr'$ of $\Cscr$ \emph{determines the happy nucleolus} if
the happy nucleolus is the unique vector $y$ that lexicographically maximizes $\exc^y[\Cscr']$ among all     
$y\in\mathbb{R}_{\ge 0}^P$ with $y(P)=\LP(P,\Tscr,c)$.
\end{definition}

For the (unhappy) nucleolus, such sets $\Cscr'$ are known as \emph{characterization sets} which were independently introduced by Granot, Granot and Zhu \cite{granot1998characterization} as well as Reijnierse and Potters \cite{reijnierse1998b}. We refer to \cite{sziklai2015computation} for an overview.

The coalitions we need to consider are all given by the elements of $\Tscr$ and the sets that result from these by removing one player:
\begin{equation*}
\Cscr^* \ := \ \left\{ (S,\{T\}) : T\in\Tscr,\, S \notin \{\emptyset,P\},\, \left(S=T \text{ or } S=T\setminus\{p\} \text{ for some } p\in T \big.\right) \Big.\right\} \enspace.
\end{equation*}
Note that $\Cscr^*\subseteq \Cscr^{\textnormal{simple}}$. 
We will show that $\Cscr^*$ always determines the happy nucleolus.
In fact, we need only efficient covers:
\begin{equation*}
\Cscr^{**} \ := \ \left\{ (S,\{T\}) \in \Cscr^* : \{T\}\in\arg\min\{c(\Uscr): \Uscr\subseteq\Tscr,\, \Uscr \text{ covers } S \} \Big.\right\} \enspace.
\end{equation*}
We remark that while the family $\Cscr^{**}$ can be smaller, it may not be easy to compute. 
Moreover, in contrast to $\Cscr^*$, the family $\Cscr^{**}$ depends on the cost function $c$.

\begin{restatable}{theorem}{mainthm} \label{thm:few_constraints_determine_nucleolus}
Let $(P,\Tscr,c)$ be a set covering instance. Then every $\bar\Cscr\supseteq \Cscr^{**}$ determines the happy nucleolus.
\end{restatable}

We prove this result in \cref{sec:few_coalitions_suffice}.
We will also show that (at least for many set systems $\Tscr$), 
among all families of simple pairs, $\Cscr^*$ is the unique minimal one for which every superset determines the happy nucleolus
for every cost function $c$. We refer to \cref{section:lowerbound} for details.

A consequence of \cref{thm:few_constraints_determine_nucleolus} is: 

\begin{restatable}{corollary}{polytimecorollary} \label{cor:happynucleolus_in_polytime}
For any given set covering instance $(P,\Tscr,c)$, the happy nucleolus can be computed in time bounded by a polynomial in $|P|$ and $|\Tscr|$.
\end{restatable}

This is in contrast to the (unhappy) nucleolus, which in particular involves finding the optimum integral solution value of the given set covering instance, 
which is NP-hard \cite{karp1972reducibility}. 
If the core is nonempty, the happy nucleolus coincides with the nucleolus. 
To the best of our knowledge, it was not known before that the nucleolus for set covering can be computed in polynomial time in this case.

\section{Few coalitions suffice}\label{sec:few_coalitions_suffice}

In this section, we prove \cref{thm:few_constraints_determine_nucleolus} and \cref{cor:happynucleolus_in_polytime}.
We begin with a key lemma.

\subsection{Nucleolus-defining constraints}

Let $Ay\le b$ be a linear inequality system with variables $y\in\mathbb{R}^P$, and let $A' y\le b'$ be a sub-system of $Ay\le b$,
arising by taking a subset of the rows (inequalities).
For a vector $y\in\mathbb{R}^P$, we consider the slack (or excess) $b_i-A_iy$ of each inequality $A_i y\le b_i$ and define
$\exc^y[A,b]$ to be the vector that has the entries of $b-Ay$, but sorted in nondecreasing order.
Analogously, $\exc^y[A',b']$ has the entries of $b'-A'y$, but sorted in nondecreasing order.

\begin{lemma}\label{lem:subsystem}
Let $K \subseteq \mathbb{R}^P$ be convex.
If a vector $y\in K$ lexicographically maximizes $\exc^y[A',b']$ among all vectors in $K$
and every row of $A$ is a linear combination of rows of $A'$ that have smaller or equal excess with respect to $y$, 
then $y$ lexicographically maximizes $\exc^y[A,b]$ among all vectors in $K$.
Moreover, $Ay$ is the same for all such vectors $y$ 
(and hence $y$ is unique if the columns of $A$ are linearly independent). 
\end{lemma}

\begin{proof}
Let $y\in K$ be a vector that lexicographically maximizes $\exc^y[A',b']$ among all vectors in $K$.
Assume that every row of $A$ is a linear combination of rows of $A'$ that have smaller or equal excess.
Let $z\in K$ be a vector that lexicographically maximizes $\exc^z[A,b]$ among all vectors in $K$.
We show $Ay=Az$.

Sort the rows of the entire system $Ay\le b$ according to their excess with respect to $y$, 
and as a tie-breaker take rows of $A'y\le b'$ first. Then $b_1-A_1y\le b_2-A_2y\le\dots$ .

Suppose $Ay\neq Az$.
Let $A_jy\le b_j$ be the first row for which $A_jy\neq A_jz$.
Then $b_i-A_iz=b_i-A_iy$ for $i<j$, and $b_i-A_iz\ge b_j-A_jy$ for all $i\ge j$ because $z$ lexicographically maximizes $\exc^z[A,b]$ among all vectors in $K$.
In particular, $b_j-A_jz > b_j-A_jy$.

First consider the case that row $j$ does not belong to the sub-system $A'y\le b'$. Then $A_j$ is a linear combination of earlier rows,
i.e., there are coefficients $\lambda_i\in\mathbb{R}$ such that $A_j=\sum_{i<j}\lambda_i A_i$.
But then $A_jz = \sum_{i<j}\lambda_i A_iz = \sum_{i<j}\lambda_i A_iy = A_jy$, a contradiction.

Now consider the case that row $j$ belongs to the sub-system $A'y\le b'$.
Let  $w=\frac{1}{2}(y+z)$. Note that $w \in K$.
We claim that $\exc^w[A',b']$ is lexicographically greater than $\exc^y[A',b']$.
To this end, observe: 
\begin{align*}
b_i-A_iw &\ = \ b_i-A_iy && \text{ for all $i<j$ } \\ 
b_j-A_jw &\ > \ b_j-A_jy && \\
b_i-A_iw &\ \ge \ b_j-A_jy && \text{ for all $i>j$ with $b_i-A_iy=b_j-A_jy$} \\
b_i-A_iw &\ > \ b_j-A_jy && \text{ for all $i>j$ with $b_i-A_iy>b_j-A_jy$} 
\end{align*}
So indeed $\exc^w[A',b']$ is lexicographically greater than $\exc^y[A',b']$, contradicting the choice of $y$.
\end{proof}

Similar arguments can be found in \cite{davis1965kernel} and \cite{kopelowitz1967computation}. 

Reijnierse and Potters \cite{reijnierse1998b} showed that \emph{a posteriori} $2|P|-2$ rows suffice and presented an example 
where a family of $2|P|-2$ rows is inclusionwise minimal. Alas, this does not imply that no family with fewer elements determines the happy nucleolus.
In Section~\ref{section:aposteriori_lowerbound}, we present a different tight example (a set covering instance with nonempty core)
and prove that no family with fewer than $2|P|-2$ sets determines the (happy) nucleolus.

However, our main focus is to prove an \emph{a priori} bound.

\subsection{Smallest family defining the happy nucleolus for set covering}\label{sec:set_cover}

The excess vector contains exponentially many entries, and computing a single entry is in general NP-hard. 
Our goal is to show that it suffices to consider an excess vector with relatively few entries that we can determine \emph{a priori} and compute easily.
As we will now prove, it actually suffices to consider at most $|P| \cdot |\Tscr|$ entries, namely those corresponding to elements of $\Cscr^{**}$.
Any superset is also fine; in particular, we can take $\Cscr^*$ because $\Cscr^{**}$ may be difficult to compute.

\mainthm*

\begin{proof}
	Let $\bar\Cscr\supseteq \Cscr^{**}$. Let $K = \{y\in\mathbb{R}_{\ge 0}^P : y(P)=\LP(P,\Tscr,c)\}$. Obviously, $K$ is convex. 
	For a coalition $S$, let $\Uscr_S\subseteq\Tscr$ such that $\Uscr_S$ covers $S$ and $c(\Uscr_S)$ is minimum (we can break ties arbitrarily).
	Let $\Cscr' := \{(S,\Uscr_S) : \emptyset\not=S\subseteq P\}$. 
	By definition, $\exc^y[\Cscr']=\exc^y$ for all $y\in K$.

	Consider the linear inequality system $Ay \le b$ that contains for each pair $(S,\Uscr)\in\Cscr$ a row $\sum_{p \in S} y_p \le c(\Uscr)$. 
	Let $A' y \le b'$ be the sub-system of $Ay\le b$ that arises by taking the rows that correspond to elements of $\Cscr'$.
	We apply \cref{lem:subsystem} to the sub-system $A'y \le b'$, the set $K$, 
	and a vector $y \in K$ lexicographically maximizing $\exc^y[\Cscr']$. 
	This is possible because every row of $A$ is equal to a row of $A'$ with smaller or equal excess.
	Hence the happy nucleolus (which lexicographically maximizes $\exc^y=\exc^y[\Cscr']$) is the unique vector
	that lexicographically maximizes $\exc^y[\Cscr]$.

	Let now $\bar A y \le \bar b$ be the sub-system of $Ay\le b$ that arises by taking the rows that correspond to the elements of $\bar\Cscr$.
	We will show that we can also apply \cref{lem:subsystem} to the sub-system $\bar A y \le \bar b$, the set $K$, 
	and a vector $y \in K$ that lexicographically maximizes $\exc^y[\bar\Cscr]$.
	This will immediately finish the proof. 
	Consider one row of $Ay \le b$, which corresponds to a pair $(S,\Uscr)\in\Cscr$. 
	Let $\Uscr_S=\{U_1,\dots, U_k\} \in \arg\min\{c(\Uscr') : \Uscr' \subseteq \Tscr,\, \Uscr' \text{ covers } S\}$, 
	and without loss of generality let $\Uscr_S$ be minimal (i.e., no subset of $\Uscr_S$ covers $S$).
	Let $S_i:= (S\cap U_i)\setminus(U_1\cup\cdots\cup U_{i-1})$ for $i=1,\ldots,k$; so $\{S_1,\ldots,S_k\}$ is a partition of $S$. 
	
	Let $y \in K$ lexicographically maximize $\exc^y[\bar\Cscr]$. 
	For the excess of $(S, \Uscr)$, we get
	\begin{equation} \label{eq:boundexcessbycover}
		c(\Uscr)-y(S) \ \ge \ c(\Uscr_S)-y(S) \ = \ \sum_{i=1}^k (c(U_i) - y(S_i)) \enspace.
	\end{equation}
	We claim that all terms on the right-hand side of~\eqref{eq:boundexcessbycover} are nonnegative.
	Suppose not. Then $c(U_i) - y(U_i) \le c(U_i) - y(S_i) < 0$ for some $i$, which means $U_i\not=P$.
	By \cref{prop:lp_allocated} and $(U_i,\{U_i\})\in\Cscr^{**}\subseteq\bar\Cscr$, 
	the fact that $y$ lexicographically maximizes $\exc^y[\bar\Cscr]$ implies $c(U_i) - y(U_i) \geq 0$, a contradiction.
	 
	Since all terms on the right-hand side of~\eqref{eq:boundexcessbycover} are nonnegative, 
	the excess of $(S_i,\{U_i\})$ is smaller than or equal to the excess of $(S,\Uscr)$. 
	Obviously, $A_S$ is a linear combination (actually: the sum) of the vectors $A_{S_i}$, 
	where $A_R$ for $R \subseteq P$ denotes the incidence vector of $R$. 
	
	Hence, to conclude the proof, we will show that each $A_{S_i}$ is a linear combination of rows of $\bar A$ 
	with excess smaller than or equal to the excess $c(U_i) -y(S_i)$ of $(S_i, \{U_i\})$. 
	Note that for each $i=1,\dots,k$, the vector $A_{S_i}$ can be written as
	\[
	A_{S_i} \ = \ \sum_{p \in U_i \setminus S_i} A_{U_i \setminus \{p\}} - \left(|U_i \setminus S_i | -1 \right) \cdot A_{U_i} \enspace.
	\]
	Moreover, for each row $A_R$ on the right-hand side of the equation, 
	$S_i \subseteq R$ and thus $(R, \{U_i\})$ has no larger excess than $(S_i, \{U_i\})$.
	Because $(R, \{U_i\})$ belongs to $\Cscr^{**}$ and hence to $\bar\Cscr$, this finishes the proof.
\end{proof}

Note that the family that consists only of the pairs $(T,\{T\})$ for $T\in\Tscr$ 
does not always determine the happy nucleolus,
as the example $P=\{1,2\}$, $\Tscr=\{\{1\},\{1,2\}\}$ with $c(T)=1$ for all $T\in\Tscr$ shows.

\subsection{Computing the happy nucleolus for set covering} \label{section:polytimecomputation}

\cref{thm:few_constraints_determine_nucleolus} reduces the effort in the MPS scheme
because fewer coalitions need to be considered, and only coalitions for which computing the excess with respect to some given $y \in \mathbb{R}^P$ is trivial:
$\Cscr^{**}\subseteq\Cscr^*\subseteq\Cscr^{\textnormal{simple}}$. 
More precisely, we can prove:

\polytimecorollary*

\begin{proof}
    By \cref{thm:few_constraints_determine_nucleolus}, we can apply the MPS scheme to $\Cscr^*$ to determine the happy nucleolus. 
    It remains to show that we can implement the MPS scheme in polynomial time. 
    In each iteration of the MPS scheme, we solve the following linear program:
    \begin{equation}\label{eq:maschler_scheme_lp}
    	\begin{array}{lrcll}
    		\max \ \varepsilon \\[1mm]
    		\text{subject to } & y(P) & = & \LP(P,\Tscr,c) \\
    		& y(S)  & \le & c(T) - \varepsilon& \forall (S, \{T\}) \in \Cscr^* \setminus \Fscr \\
    		& y(S) & = & c(T) - \varepsilon_{(S,\{T\})} & \forall (S, \{T\}) \in \Fscr \\
    		& y &\ge & 0 \enspace,
    	\end{array}
    \end{equation}
    where $\Fscr$ contains all pairs $(S,\{T\}) \in \Cscr^*$ whose excess values have already been fixed in earlier iterations. 
    For such a pair $(S,\{T\}) \in \Fscr$, $\varepsilon_{(S,\{T\})}$ denotes its fixed excess. 
    Let $\varepsilon^*$ denote the optimum value of \eqref{eq:maschler_scheme_lp}. 
    The main difficulty is to determine the set of pairs in $\Cscr^* \setminus \Fscr$ whose excess is equal to $\varepsilon^*$ for every optimum LP solution.
    More precisely, we need to find at least one element of this set. 
    For this, we consider the dual linear program:
    \begin{equation}\label{eq:maschlers_scheme_dual_lp}
    	\begin{array}{lrcll}
    	\multicolumn{5}{l}{	\min \ \displaystyle \sum_{(S,\{T\}) \in \Cscr^*} c(T) \cdot z_{(S,\{T\})} \ - \sum_{(S,\{T\}) \in \Fscr} \varepsilon_{(S,\{T\})} \cdot z_{(S,\{T\})} \ + \ \delta \cdot \LP(P,\Tscr,c) } \\[9mm]
    		\text{subject to } & \displaystyle \sum_{(S,\{T\}) \in \Cscr^*\setminus\Fscr} z_{(S,\{T\})}  & \ge & 1 \\[7mm]
    		& \delta + \displaystyle \sum_{(S,\{T\}) \in \Cscr^*: p \in S} z_{(S,\{T\})}  & \ge & 0 & \forall p \in P \\[7mm]
    		& z_{(S,\{T\})} &\ge & 0 & \forall (S, \{T\}) \in \Cscr^* \setminus \Fscr \enspace.
    	\end{array}
    \end{equation}
    
    By a well-known result by \'{E}va Tardos \cite{tardos1986lp}, we can solve  \eqref{eq:maschler_scheme_lp} and \eqref{eq:maschlers_scheme_dual_lp} in time  polynomial in the number of rows and columns of the constraint matrix because  all entries in the constraint matrix are either $0$ or $1$.
    Hence, since $|\Cscr^*| \le |P| \cdot |\Tscr|$, we can solve \eqref{eq:maschlers_scheme_dual_lp} in time bounded by a polynomial in $|P|$ and $|\Tscr|$.
    
    Let $(z^*, \delta^*)$ be an optimum solution to \eqref{eq:maschlers_scheme_dual_lp}.
    By complementary slackness, for each $(S,\{T\}) \in \Cscr^*\setminus\Fscr$ with $z^*_{(S,\{T\})} > 0$ and each optimum solution $(y^*, \varepsilon^*)$ to \eqref{eq:maschler_scheme_lp}, the excess of $(S,\{T\})$ with respect to $y^*$ is exactly $\varepsilon^*$. 
    Hence, we can add all $(S,\{T\}) \in \Cscr^*\setminus\Fscr$ with $z^*_{(S,\{T\})} > 0$ to $\Fscr$ and set $\varepsilon_{(S,\{T\})}$ to $\varepsilon^*$. 
    Note that by the first inequality in \eqref{eq:maschlers_scheme_dual_lp}, there must be at least one such  $(S,\{T\}) \in \Cscr^*\setminus\Fscr$. 
    Hence, the total number of iterations of the MPS scheme is bounded by $|\Cscr^*|$.
\end{proof}

\section{Minimality of the family determining the happy nucleolus} \label{section:lowerbound}

In this section, we would like to show that \cref{thm:few_constraints_determine_nucleolus} is, in a certain sense, best possible.
To this end, we will first prove minimality of $\Cscr^*$ under some conditions (cf.\ \cref{thm:minimality}),
and then we discuss why these conditions are both reasonable and necessary.

\subsection{Minimality of \bm{$\Cscr^*$}}

For many set systems, $\Cscr^*$ is the unique minimal subfamily
$\Fscr$ of $\Cscr^{\textnormal{simple}}$ such that $\Fscr$ and every super-family of $\Fscr$
determines the happy nucleolus for all cost functions (see \cref{section:discussstability} for a discussion). 
To make precise what we mean by ``many set systems'', 
we say that a set system $\Tscr\subseteq 2^P$ satisfies the \emph{complement condition} if
for every $T \in \mathcal{T}$ there exists a cover $\Uscr \subseteq \mathcal{T} \cap 2^{P \setminus T} $ of its complement $P \setminus T$.
Note that the complement condition is satisfied naturally in many applications (e.g., when $\Tscr$ contains all singletons), including vehicle routing.
We refer to \cref{section:discussiononminimality} for further discussion.

\begin{theorem} \label{thm:minimality}
Let $P$ be a finite set and $\Tscr\subseteq 2^P$ a set system that covers $P$ and satisfies the complement condition.
Then, for every $(S^*,\Uscr^*)\in\Cscr^*$, there exists a cost function $c:\Tscr\to\mathbb{R}_{\ge 0}$ such that $\Cscr^{\textnormal{simple}}\setminus\{(S^*,\Uscr^*)\}$ does not determine the happy nucleolus.
\end{theorem}

\begin{proof}
    Let $n=|P|$ and $\Tscr$ as in the theorem and $(S^*,\Uscr^*)\in \Cscr^*$. By definition of $\Cscr^*$, we have $\Uscr^* = \{T^*\}$ for some $T^* \in \Tscr$ and will denote $k=|T^*|$. The coalition $S^*$ must look like one of the following:
    \begin{enumerate}
	   \item[(a)] $S^* = T^* \setminus \{p^*\}$ for some $p^* \in T^*$; or
	   \item[(b)] $S^* = T^*$.
	\end{enumerate}
We consider these cases separately.

\smallskip\noindent\textbf{(a):}
Then $k\ge 2$. We define the cost function
        \begin{equation*}
            c(T) = \begin{cases}
                k &\text{for } T = T^*  \\
                0 &\text{for } T \in \Tscr \cap 2^{P \setminus T^*} \\
                k+2 &\text{otherwise}.
            \end{cases}
        \end{equation*}

        By the complement condition, we can cover $P \setminus T^*$ with cost 0. 
        Thus, every cost allocation with nonnegative excess vector must be 0 on $P \setminus T^*$. 
        The happy nucleolus $y^*$ now distributes the cost of $T^*$ equally among its members, as this maximizes the minimal nonzero excess 
        (which is attained at each coalition $T^* \setminus \{p\}$ for some $p \in T^*$). 
        Hence, $y^*_{p} = 1$ for every $p \in T^*$. 
        Lexicographically maximizing $\exc^y[\Cscr^{\textnormal{simple}}\setminus\{(S^*,\Uscr^*)\}]$ yields a different optimum solution. 
        Consider $y'$ with $y'_{p^*}=0$ and $y'_{p}=\frac{k}{k-1}$ for every $ p \in T^*\setminus\{p^*\}$. 
        Then $\exc^{y'}[\Cscr^{\textnormal{simple}}\setminus\{(S^*,\Uscr^*)\}]$ is lexicographically larger than $\exc^{y^*}[\Cscr^{\textnormal{simple}}\setminus\{(S^*,\Uscr^*)\}]$:
        
        \begin{center}
        { \doublespacing \scriptsize
        \begin{tabular}{|l!{\vrule width 1.1pt}c|c|c!{\vrule width 1.1pt}c|}
            \hline
            \rowcolor{gray!10} \textbf{Excess of \bm{$S$} for...} & \bm{$S\cap T^*\in\{\emptyset,T^*\}$} & \bm{$S\cap T^*=T^*\setminus\{p^*\}$} & \bm{$S\cap T^*=T^*\setminus\{p\}, p \neq p^*$}  & \textbf{Other coalitions} \\ \noalign{\hrule height 1.1pt}
            \cellcolor{gray!10} \bm{$y^*$} \textbf{(happy nucl.)}& 0 & 1 & 1 &  $\geq 2$\\ \hline
            \cellcolor{gray!10} \bm{$y'$} & 0 & \cellcolor{red!20} $0$ & \cellcolor{green!20} $\frac{k}{k-1} > 1$ &  $\geq \frac{k}{k-1}$ \\ \hline
        \end{tabular}
        }   
        \end{center}
        
\smallskip\noindent\textbf{(b):}
By definition of $\Cscr^*$, we know that $T^* \neq P$ in this case. Choose an arbitrary player $p^* \in P \setminus T^*$. 
        We define the cost function
        \begin{equation*}
            c(T) = \begin{cases}
                k &\text{for } T = T^* \\
                \frac{1}{n} &\text{for } p^* \in T \in \Tscr \cap 2^{P \setminus T^*} \\
                0 &\text{for } p^* \notin T \in \Tscr \cap 2^{P \setminus T^*} \\
                k+2 &\text{otherwise} \enspace.
            \end{cases}
        \end{equation*}

    By the complement condition, we know that the minimum cost $x$ to cover $P \setminus T^*$ fulfills $\frac{1}{n} \leq x < 1$. 
    By complementary slackness (see the LP duality equation in the proof of \cref{prop:lp_allocated}),
    the happy nucleolus must have excess 0 on all sets of an optimum fractional set covering solution. 
    This applies to $T^*$ and at least one other set $T'$ with  $p^* \in T' \in \Tscr \cap 2^{P \setminus T^*}$.
    Thus, the happy nucleolus $y^*$ must fulfill $\theta^{y^*}(T^*)=0$ and $\theta^{y^*}(T')=0$.
    Lexicographically maximizing $\exc^y[\Cscr^{\textnormal{simple}}\setminus\{(S^*,\Uscr^*)\}]$ yields a different optimum solution. 
    Consider $y'$ with $y'_{p}=0$ for every $p \in P \setminus T^*$ and $y'_{p}=1 + \frac{x}{k}$ for every $p \in T^*$.
    The existence of $T'$ shows that $\exc^{y'}[\Cscr^{\textnormal{simple}}\setminus\{(S^*,\Uscr^*)\}]$ is lexicographically larger than $\exc^{y^*}[\Cscr^{\textnormal{simple}}\setminus\{(S^*,\Uscr^*)\}]$:

    \begin{center}
        { \doublespacing \scriptsize
        \begin{tabular}{|l!{\vrule width 1.1pt}c|c|c!{\vrule width 1.1pt}c|}
            \hline
            \rowcolor{gray!10} \textbf{Excess of \bm{$S$} for...} & \bm{$S=T^*$} & \bm{$p^* \notin S \subseteq P \setminus T^*$} & \bm{$S=T'$} & \textbf{Other coalitions} \\ \noalign{\hrule height 1.1pt}
            \cellcolor{gray!10} \bm{$y^*$} \textbf{(happy nucl.)}& 0 & 0 & 0 & $\geq 0$\\ \hline
            \cellcolor{gray!10} \bm{$y'$} & \cellcolor{red!20} $-x$ & 0 & \cellcolor{green!20} $\frac{1}{n} > 0$ & $\geq \frac{1}{n}$ \\ \hline
        \end{tabular}
        }   
        \end{center}
    
    \qedhere
\end{proof}

In the remainder of this section, we discuss the statement of \cref{thm:minimality} in more detail.

\subsection{Stability of \bm{$\Cscr^*$}\label{section:discussstability}}

\begin{figure}[h]
    \begin{center}
        \begin{tikzpicture}[scale=1.5, thick]
            \node[circle, fill, draw, scale = 0.5, label=above:{\scriptsize $p_1$}] (1) at (1,0) {}; 
            \node[circle, fill, draw, scale = 0.5, label=above:{\scriptsize $p_2$}] (2) at (2,0) {};
            \node[circle, fill, draw, scale = 0.5, label=above:{\scriptsize $p_3$}] (3) at (3,0) {};
    
            \draw[color=blue] (2,0.09) ellipse (40pt and 15pt) node[fill=white, above=15pt] {3};
        \end{tikzpicture}
    \end{center}
    \caption{Instance with three players and only one set, covering them all. Equally distributing $y_p =1$ for all players $p$ yields the happy nucleolus. 
    It is enough to consider singleton coalitions to determine it. But adding a 2-element coalition (and not all 2-element coalitions) will yield a different solution.
	\label{fig:stability}}
\end{figure}
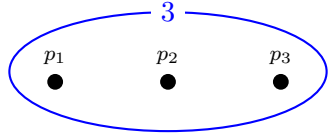

\cref{thm:minimality} shows that removing an arbitrary element of $\Cscr^*$ from $\Cscr^{\textnormal{simple}}$ 
can lead to a different solution than the happy nucleolus when lexicographically maximizing the resulting excess vector. 
Nevertheless, removing even more elements might again yield the happy nucleolus (see \cref{fig:stability}). 
We are not interested in these smaller subfamilies that lead to the happy nucleolus coincidentally because they lack a natural stability property:

$\Cscr^*$ is stable under addition of pairs by \cref{thm:few_constraints_determine_nucleolus}: any superset determines the happy nucleolus. 
No other subfamily of $\Cscr^{\textnormal{simple}}$ can have this stability property for all cost functions by \cref{thm:minimality}.

\subsection{Discussion of the complement condition}\label{section:discussiononminimality}

Now we explain the complement condition, which requires that
for every $T \in \mathcal{T}$ there exists a cover $\Uscr \subseteq \mathcal{T} \cap 2^{P \setminus T}$ of its complement $P \setminus T$.

The complement condition cannot simply be omitted as requirement in \cref{thm:minimality}: 
for a general set system $\Tscr$, we cannot always find a cost function $c$ for which the happy nucleolus needs all pairs in $\Cscr^*$
(even when restricting to $\Cscr^{\textnormal{simple}}$). This is demonstrated by the following example.

\begin{example} \label{example:triangle}
Consider the triangle on $P=\{p_1,p_2,p_3\}$, 
i.e., $\Tscr$ contains all two-element coalitions (\cref{fig:triangle}).
Note that the complement condition does not hold. 

\begin{figure}[h]
    \begin{center}
        \begin{tikzpicture}[xscale=1.5, yscale=1.3, thick]
            \node[circle, fill, draw, scale = 0.5, label={[yshift=-0.05cm]{\scriptsize $p_1$}}] (1) at (1,0) {}; 
            \node[circle, fill, draw, scale = 0.5, label={[yshift=-0.05cm]{\scriptsize $p_2$}}] (2) at (3,0) {};
            \node[circle, fill, draw, scale = 0.5, label={[yshift=-0.05cm]{\scriptsize $p_3$}}] (k) at (2,1.73) {};
    
            \draw[color=blue] (2,0) ellipse (40pt and 15pt) node[fill=white, below=10pt] {$T_1$};
    
            \draw[color=green!50!black, rotate around={60:(1,0)}] (2,0) ellipse (40pt and 15pt) node[fill=white, left=15pt, above=2pt] {$T_2$};
    
            \draw[color=red!80!black, rotate around={-60:(3,0)}] (2,0) ellipse (40pt and 15pt) node[fill=white, right=15pt, above=2pt] {$T_3$};

            rotate around={45:(5,5)}
    
        \end{tikzpicture}
    \end{center}
    \caption{Triangle instance of \cref{example:triangle}. 
    For any cost function $c:\{T_1,T_2,T_3\}\to\mathbb{R}_{\ge 0}$, there is only one cost allocation $y\in\mathbb{R}^P_{\ge 0}$ 
    with $y(P)=\LP(P,\Tscr,c)$ that makes the three coalitions $T_1,T_2,T_3$ happy.
	\label{fig:triangle}}
\end{figure}
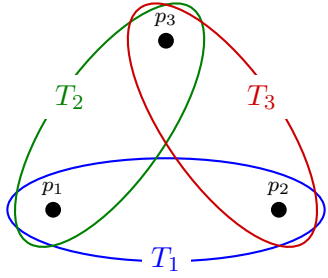

We can assume that $c(T_1) \leq c(T_2) \leq c(T_3)$. If $c(T_3) < c(T_1) + c(T_2)$, the unique optimum fractional cover takes each set with value $\frac12$. Otherwise, $T_1$ and $T_2$ constitute an optimum (fractional) cover. 
Thus, we can immediately identify the unique vector $y^*$ with nonnegative excess for each coalition and $y^*(P)=\LP(P,\Tscr,c)$, which must be the happy nucleolus:
\begin{align*}
    y^*_{p_1} &\ = \ \max \left\{0 \enspace , \enspace \frac{c(T_1) + c(T_2) - c(T_3)}{2} \right\} \enspace, \\[1mm]
    y^*_{p_2} &\ = \ c(T_1) - y^*_{p_1} \enspace, \\[1mm]
    y^*_{p_3} &\ = \ c(T_2) - y^*_{p_1} \enspace.
\end{align*}

Note that $y^*$ is also the unique vector with $y^*(P)=\LP(P,\Tscr,c)$ and nonnegative excess for each pair $(T,\{T\})$ with $T \in \Tscr$. 
Hence, every $\bar\Cscr \supseteq \left\{ (T,\{T\}) : T \in \Tscr \right\}$ determines the happy nucleolus,
and the pairs $(\{p_i\}, \{T_j\}) \in \Cscr^*$ are not needed. Hence \cref{thm:minimality} does not apply.  
\end{example}

Moreover, we remark that the complement condition holds naturally in many applications, including vehicle routing.
Even more, for an arbitrary set covering instance $(P,\Tscr, c)$, we can always define an equivalent 
instance $(P,\Tscr',c')$ that fulfills the complement condition by adding all singleton sets with their minimum covering cost.
Neither the set covering instance nor its happy nucleolus change by this polynomial-time transformation.

\section{Lower bound for families determining the (happy) nucleolus a posteriori \label{section:aposteriori_lowerbound}}

Reijnierse and Potters \cite{reijnierse1998b} showed that \emph{a posteriori} $2|P|-2$ coalitions suffice to determine the nucleolus.
They presented an example where a family of $2|P|-2$ coalitions determines the nucleolus and no element can be removed from
this family without destroying this property. This however does not show that no smaller family determines the nucleolus.
We present a different tight example (which is a set covering instance with nonempty core) 
in which no family with fewer than $2|P|-2$ coalitions does the job:

\begin{theorem} \label{thm:aposteriorilowerbound}
For every $n \in \mathbb{N}$, there is a set covering instance with $n$ players and nonempty core such that 
no family $\Cscr'\subseteq\Cscr$ with fewer than $2n-2$ elements determines the (happy) nucleolus.
\end{theorem}

\begin{proof}
	Let $n \in \mathbb{N}$. 
	Let $P=\{p_1,\ldots,p_n\}$, and let $\Tscr = \{\{p_1,\dots,p_i\}: 1 \le i \le n\}$ with $c(\{p_1,\dots, p_i\}) = i$ for each $1 \le i \le n$. 
	See \cref{fig:chain} for an illustration. 
	
	Let $y \in \mathbb{R}_{\ge 0}^P$ be the happy nucleolus. We write $y_i = y_{p_i}$. Then,
	\begin{equation} \label{eq:happynucleolus_chainexample}
		\begin{aligned}
			y_i&\ = \ 1-2^{-i} \hspace{2cm} \text{ for } i=1,\ldots,n-1 \enspace,\\[-1mm] 
			y_n&\ = \ 2-2^{-(n-1)} \enspace.
		\end{aligned}
	\end{equation}
	
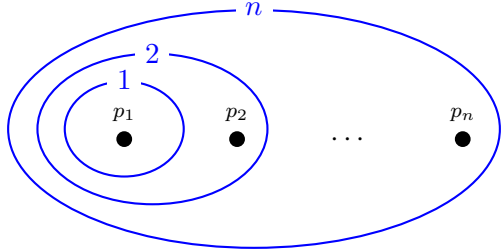
\begin{figure}[ht]
\begin{center}
    \begin{tikzpicture}[scale=1.5, thick]
        \node[circle, fill, draw, scale = 0.5, label=above:{\scriptsize $p_1$}] (1) at (1,0) {}; 
        \node[circle, fill, draw, scale = 0.5, label=above:{\scriptsize $p_2$}] (2) at (2,0) {};
        \node [] (dots1) at (3,0) {$\cdots$};
        \node[circle, fill, draw, scale = 0.5, label=above:{\scriptsize $p_n$}] (j) at (4,0) {};

        \draw[color=blue] (1,0.09) ellipse (15pt and 12pt) node[fill=white, above=11pt] {1};
        \draw[color=blue] (1.25,0.09) ellipse (29pt and 19pt) node[fill=white, above=21pt] {2};
        \draw[color=blue] (2.15,0.09) ellipse (62pt and 30pt) node[fill=white, above=39pt] {$n$};
    \end{tikzpicture}
\end{center}
\caption{\small Illustration of the set covering instance in the proof of Theorem~\ref{thm:aposteriorilowerbound},
with players $p_1,\dots, p_n$. For each $i=1,\dots, n$, we have a set $\{p_1, \dots, p_i\}$ of cost $i$.
\label{fig:chain}}
\end{figure}

This can be obtained by the MPS scheme (cf.\ \cref{fig:chain_maschler}):
The excess is of course 0 for the entire set~$P$. 
The next smallest excess cannot be larger than $\frac{1}{2}$, and this value is attained by the coalitions $\{p_1\}$ and $P\setminus\{p_1\}$; 
this determines $y_1=\frac{1}{2}$. 
With $y_1$ fixed, the MPS scheme maximizes the next smallest excess to be $\frac{3}{4}$, 
attained by the coalitions $\{p_1,p_2\}$ and $P\setminus\{p_2\}$; this determines $y_2=\frac{3}{4}$.
With $y_1$ and $y_2$ fixed, the MPS scheme maximizes the next smallest excess to be $\frac{7}{8}$, 
attained by the coalitions $\{p_1,p_2,p_3\}$ and $P\setminus\{p_3\}$; this determines $y_3=\frac{7}{8}$.
The MPS scheme continues in this manner until maximizing the excess of the coalitions $\{p_1,p_2,p_3,\ldots,p_{n-1}\}$ and $P\setminus\{p_{n-1}\}$ 
to be  $1-2^{-(n-1)}$; this determines $y_{n-1}=1-2^{-(n-1)}$ (and hence also determines $y_n$).

\begin{figure}[h]

\begin{center}
\scalebox{.95}
{
\scriptsize
\doublespacing
\begin{tabular}{|l!{\vrule width 1.1pt}c|c!{\vrule width 1.1pt}c|c!{\vrule width 1.1pt}c|c!{\vrule width 1.1pt}c!{\vrule width 1.1pt}c|c|}
    	\hline 
    	&
	\multicolumn{2}{c!{\vrule width 1.1pt}}{\cellcolor{blue!20}\textbf{Stage 1}} & \multicolumn{2}{c!{\vrule width 1.1pt}}{\cellcolor{blue!15} \textbf{Stage 2}} & \multicolumn{2}{c!{\vrule width 1.1pt}}{\cellcolor{blue!10} \textbf{Stage 3}} & $\cdots$ & \multicolumn{2}{c|}{\cellcolor{blue!5}  \textbf{Stage} \bm{$n-1$}}\\ \hline
    	\textbf{Coalition} &
    	\cellcolor{blue!20} $\{p_1\}$ & \cellcolor{blue!20} $P\setminus \{p_1\}$ & \cellcolor{blue!15} $\{p_1, p_2\}$ & \cellcolor{blue!15} $P\setminus \{p_2\}$ & \cellcolor{blue!10} $\{p_1, p_2, p_3\}$ & \cellcolor{blue!10} $P\setminus \{p_3\}$ & $\cdots$ & \cellcolor{blue!5} $\{p_1, \ldots , p_{n-1}\}$ & \cellcolor{blue!5} $P\setminus \{p_{n-1}\}$ \\ \hline
    	\textbf{Excess} &
    	\multicolumn{2}{c!{\vrule width 1.1pt}}{\cellcolor{blue!20} $\frac{1}{2}$} & \multicolumn{2}{c!{\vrule width 1.1pt}}{\cellcolor{blue!15} $\frac{3}{4}$} & \multicolumn{2}{c!{\vrule width 1.1pt}}{\cellcolor{blue!10} $\frac{7}{8}$} & $\cdots$ & \multicolumn{2}{c|}{\cellcolor{blue!5} $1-2^{-(n-1)}$}\\ \hline
\end{tabular}
}
\end{center}
\caption{\small Stages in the MPS scheme for the set covering instance in \cref{fig:chain}. \label{fig:chain_maschler}}
\end{figure}

	
	The happy nucleolus $y$ has the following property that we will exploit:
	\begin{equation} \label{eq:coaltionswithsameexcess}
	\begin{aligned}
	& \text{If $A$ and $B$ are two different coalitions such that $y(A)-y(B)$ is an integer,} \\[-1mm]
	& \text{then there is an index $j\in\{1,\ldots,n-1\}$ such that $\left\{\big. A\setminus B,\, B\setminus A \right\} = \left\{\big. \{p_j\},\, \{p_{j+1},\ldots,p_n\}\right\}$.}
	\end{aligned}
	\end{equation}
This follows easily from \eqref{eq:happynucleolus_chainexample}.

	Now let $\Cscr'\subseteq\Cscr$ such that $\Cscr'$ determines the happy nucleolus. We show $|\Cscr'|\ge 2n-2$.
	For $i=1,\ldots,n-1$, define vectors $y^{i+}=y+\epsilon(\chi^i-\chi^{i+1})$ and $y^{i-}=y-\epsilon(\chi^i-\chi^{i+1})$, 
	where $\chi^i$ denotes the $i$-th unit vector and $\epsilon$ is a sufficiently small positive constant. 
	Let $A_i\in\Cscr'$ be a set that contains $i$ but not $i+1$ and has smallest excess with respect to $y$ among such sets.
	Let $B_i\in\Cscr'$ be a set that contains $i+1$ but not $i$ and has smallest excess with respect to $y$ among such sets.
	Such sets must exist, and their excess with respect to $y$ must be the same, 
	because $y^{i+}$ and $y^{i-}$ both have a lexicographically smaller excess vector with respect to $\Cscr'$. 
	Since $A_i$ and $B_i$ have the same excess, $y(A_i)-y(B_i)$ is an integer. 
	By \eqref{eq:coaltionswithsameexcess}, $A_i$ and $B_i$ agree on players $\{p_1,\ldots,p_{i-1}\}$ but disagree on player $p_i$.
	Hence, the $2n-2$ sets $A_i,B_i$ ($i=1,\ldots,n-1$) in $\Cscr'$ are all distinct. 
\end{proof}

\bibliographystyle{plain}
\bibliography{refs}

\end{document}